\documentclass[final]{siamltex}
\usepackage{amsmath,amsfonts, amssymb, color}

\newtheorem{remark}[theorem]{Remark}

\newcommand{\ang}[1]{\langle  #1 \rangle }  

\newcommand{\tr}{{\rm tr}}
\newcommand{\p}{\mathbb{P}}
\newcommand{\Q}{\mathbb{Q}}

\newcommand{\E}[1]{\mathbb{E}\left[#1\right]}
\newcommand{\Et}[1]{\mathbb{E}_t\left[#1\right]}

\newcommand{\trieq}{\stackrel{\triangle}{=}}

\newcommand{\cF}{{\mathcal F}}
\newcommand{\cG}{{\mathcal G}}
\newcommand{\cS}{{\mathcal S}}
\newcommand{\R}{\mathbb{R}}
\newcommand{\id}{{\mathbf 1}}
\newcommand{\as}{\mbox{{\rm a.s.}}}

\title{{Time-Inconsistent Stochastic Linear--Quadratic Control}}

\author{Ying Hu\thanks{IRMAR,
Universit\'e Rennes 1, 35042 Rennes Cedex, France. This author is  partially supported by
the Marie Curie ITN Grant, ``Controlled Systems'', GA no.213841/2008.} \and
Hanqing Jin\thanks{Mathematical Institute and Nomura Centre for Mathematical Finance, and
Oxford--Man Institute of Quantitative Finance, The University of
Oxford, 24--29 St Giles, Oxford OX1 3LB, UK. This author is  partially supported by  research
grants from the Nomura Centre for
Mathematical Finance and the Oxford--Man Institute of Quantitative Finance.} \and
Xun Yu Zhou\thanks{Mathematical Institute and Nomura Centre for Mathematical Finance, and
Oxford--Man Institute of Quantitative Finance, The University of
Oxford, 24--29 St Giles, Oxford OX1 3LB, UK, and Department of
Systems Engineering and Engineering Management, The Chinese
University of Hong Kong, Shatin, Hong Kong. This author is supported by  a start-up fund of the University of Oxford, research
grants from the Nomura Centre for
Mathematical Finance and the Oxford--Man Institute of Quantitative Finance, and an GRF grant \#CUHK419511.}}

\begin{document}

\maketitle

\begin{abstract}
In this paper, we formulate a general time-inconsistent stochastic linear--quadratic (LQ) control problem.
The time-inconsistency arises from the presence of a quadratic term of the expected state as well as  a state-dependent term in the objective functional.
We define an equilibrium, instead of optimal, solution  within the class of open-loop controls, and derive a sufficient condition for equilibrium controls
via
a flow  of forward--backward
stochastic differential equations. When the state is one dimensional and the coefficients in the problem are all deterministic,
we find an explicit equilibrium control. As an application, we then consider a mean-variance portfolio selection model
in a complete financial market where the risk-free rate is a deterministic function of time but all the other market parameters are possibly stochastic
processes. Applying the general sufficient condition,
we obtain explicit equilibrium strategies when the risk premium  is both deterministic and stochastic.
\end{abstract}



\pagestyle{myheadings}
\thispagestyle{plain}
\markboth{Ying Hu, Hanqing Jin and Xun Yu Zhou}{Time-Inconsistent Stochastic LQ Control}

\section{Introduction}\label{introduction}

Stochastic control is now a mature and well established subject of study \cite{FS,YZ}. Though not explicitly stated at most of the times, a standing assumption
in the study of  stochastic control  is the time consistency, a fundamental property of
conditional expectation with respect to a progressive filtration. As a result, an optimal control viewed from today will remain optimal viewed from tomorrow.
Time-consistency provides the theoretical foundation of the dynamic programming approach including the resulting HJB equation, which is in turn
a pillar of the modern stochastic control theory.

However, there are overwhelmingly more time-inconsistent problems than their time-consistent counterparts. Hyperbolic discounting \cite{A, LP}
 and  continuous-time mean--variance portfolio selection model \cite{ZL,Basak} provide
two well-known  examples of time-inconsistency. Probability distortion, as in behavioral finance models \cite{JZ}, is yet another distinctive source of
time-inconsistency.

One way to get around  the time-inconsistency issue is to consider only  pre-committed controls (i.e., the controls are optimal only when viewed at the initial time);
see, e.g., \cite{ZL} and all the follow-up works to date on the Markowitz problem, as well as \cite{JZ} on the behavioral portfolio choice problem.
While these controls are of practical and theoretical value, they have not really addressed the time-inconsistency nor provided solutions in a dynamic sense.

Motivated by practical applications especially in mathematical finance, time-inconsistent control problems have recently  attracted considerable research interest
and efforts attempting to seek equilibrium, instead of optimal, controls. At a conceptual level, the idea is that a decision the controller makes at every  instant
of time is considered as a game against all the decisions the future incarnations of the controller are going to make. An ``equilibrium" control is therefore one
such that any deviation from it at any time instant will be worse off. Taking this game perspective, Ekeland and Lazrak \cite{EL} approach the (deterministic)
time-inconsistent optimal control, and Bj\"ork and Murgoci \cite{BM} and Bj\"ork, Murgoci and Zhou \cite{BMZ}  extend the idea to the stochastic setting, derive an
(albeit very complicated) HJB equation, and apply the theory to a dynamic Markowitz problem. Yong \cite{Y} investigate a time-inconsistent deterministic
linear--quadratic control problem and derive equilibrium controls via some integral equations.
However,  study of time-inconsistent control is, in general,  still in its infancy.

In this paper we formulate a general stochastic linear--quadratic (LQ) control problem, where the objective functional
 includes both a quadratic term of the expected state and a state-dependent term. These non-standard terms each introduces time-inconsistency into
 the problem in somewhat different ways. Different from most of the existing literature \cite{EL,BM,BMZ,Y} where an equilibrium control is defined
 within the class of {\it feedback} controls, we define our equilibrium via open-loop controls.\footnote{Recall the class of feedback controls is a
 subset of that of open-loop ones. In standard (time-consistent) stochastic control theory, an optimal control is usually defined in the whole class of
 open-loops \cite{FS,YZ}.}  Then we derive a general sufficient condition for equilibriums through a system of forward--backward stochastic
 differential equations (FBSDEs). A intriguing feature of these FBSDEs is that a time parameter is involved; so these form a {\it flow} of
 FBSDEs. When the state process is scalar valued and all the coefficients are deterministic functions of time, we are able to reduce  this flow of FBSDEs
  into several Riccati-like ODEs, and hence obtain explicitly an equilibrium control, which turns out to be a linear feedback.

In the latter part of the paper, we study a continuous-time mean--variance portfolio selection model with state dependent
trade-off between mean and variance. A similar problem was first considered in \cite{BMZ} in the framework of feedback controls and its solution derived via
a very complicated (generalized) HJB equation. Here we allow random market parameters (hence the model and approach of \cite{BMZ} will not work)
and consider open-loop equilibriums. Applying the general sufficient condition and working through a delicate analysis, we will solve
the corresponding FBSDEs and obtain equilibrium strategies. Again, these strategies happen to be linear feedbacks.
We also compare our strategies with the ones in \cite{BMZ} when all the market coefficients are deterministic, and find that they are generally different.
This suggests that how we define equilibrium controls is critical in studying time inconsistent control problems.

The remainder of the paper is organized as follows.  The next section is devoted to the formulation of our problem
and the definition of equilibrium control. In Section \ref{formal-derivation}, we apply the spike variation technique to derive
a flow of FBSDEs and a sufficient condition of equilibrium controls. Based on this general
result, we solve in Section \ref{Deterministic-case} the case when
the state is one dimensional and all the coefficients are deterministic. In Section \ref{random-case}, we formulate a continuous-time mean--variance portfolio
selection model which is a special case of the general LQ model investigated, and derive explicitly its solution. Finally, some concluding remarks
are given in Section 6.

\section{Problem Setting}\label{problemsetting}

Let $T>0$ be the end of a finite time horizon and $(W_t)_{0 \le t \le T}=(W_t^1,\cdots,W_t^d)_{0 \le t \le T}$
a $d$-dimensional Brownian motion on a probability
space $(\Omega, \cF, \p)$. Denote by $(\cF_t)$ the
augmented filtration generated by $(W_t)$.

Throughout this paper, we use the following notation with $l$ being a generic integer:
\begin{table}[http]
\begin{tabular}{rl}
 $\mathbb S^l$:& the set of symmetric $l \times l$ real matrices.\\
  $L^2_{\cG}(\Omega; \, \R^l)$:& the set of random variables $\xi: (\Omega, \cG) \rightarrow (\R^l, {\cal B}(\R^l))$ \\ &with $\E{|\xi|^2}<+\infty$.\\
 $L^\infty_{\cG}(\Omega; \, \R^l)$:& the set of essentially bounded  random variables \\  &$\xi: (\Omega, {\cG}) \rightarrow (\R^l, {\cal B}(\R^l))$.\\
 $L^2_\cG(t, \, T; \, \R^l)$: &the set of $\{\cG_s\}_{s\in [t,T]}$-adapted processes \\ &$f=\{f_s: t\leq s\leq T\}$ with $\E{ \int_t^T|f_s|^2\, ds} < \infty$.\\
 $L^\infty_\cG(t, \, T; \, \R^l)$:&the set of essentially  bounded $\{\cG_s\}_{s\in [t,T]}$-adapted processes.\\
 $L^2_\cG(\Omega; \, C(t, \, T; \, \R^l))$:& the set of continuous 
         $\{\cG_t\}_{s\in [t,T]}$-adapted processes \\&$f=\{f_s: t\leq s\leq T\}$ with $\E{ \sup_{s\in [t,T]}|f_s|^2\, } < \infty$.\\
\end{tabular}
\end{table}

We will often use vectors and matrices in this paper, where all vectors are column vectors. For a matrix $M$,
 define
\begin{itemize}
  \item [] $M'$: Transpose of a matrix $M$.
  \item [] $|M|=\sqrt{\sum_{i,j}m_{ij}^2}$: Frobenius norm of a matrix $M$.
\end{itemize}
For a square matrix $M$, we define
 $\cS(M)=\frac{1}{2}(M+M')$ as the symmetrization of  $M$, and $\tr(M)=\sum_i M_{ii}$ as the trace of $M$.
 For a symmetric matrix $M$, we write $M\succeq 0$ if $M$ is positive semi-definite, and $M\succ 0$ if $M$ is
 positive definite.

\medskip

We consider  a continuous-time,  $n$-dimensional non-homogeneous linear controlled system
\begin{equation}\label{controlgeneral}
dX_s=[A_sX_s+B_s'u_s+b_s]ds+ \sum_{j=1}^d[C_s^jX_s+D_s^{j} u_s+\sigma_s^{j} ]dW_s^{j};\quad X_0=x_0.
\end{equation}
Here $A$ is a bounded deterministic  function on $[0, T]$ with value in $\R^{n\times n}$.
The other parameters $B,C^j,D^j$ are all essentially bounded adapted processes on $[0,T]$
with values in $\R^{ l\times n}$,
$\R^{n\times n}$, $\R^{ n\times l}$,  respectively; $b$ and $\sigma^j$ are stochastic processes
in  $L^2_\cF(0,T; \R^n)$.
The process $u\in L^2_\cF(0, \, T; \, \R^l)$ is the control, and
$X$ is the state process valued in $\R^n$.
Finally $x_0\in \R^n$ is the initial state.
It is obvious that for any control $u\in L^2_\cF(0, \, T; \, \R^l)$, there exists
a unique solution $X\in L^2_\cF(\Omega; \, C(0, \, T; \, \R^n))$.

As time evolves, we need to  consider the controlled system
starting from  time $t\in [0,T]$ and state $x_t\in L^2_{\cF_t}(\Omega; \, \R^n)$:
\begin{equation}\label{controlgeneral:t}
dX_s=[A_sX_s+B_s'u_s+b_s]ds+ \sum_{j=1}^d[C_s^jX_s+D_s^{j} u_s+\sigma_s^{j} ]dW_s^{j};\quad X_t=x_t.
\end{equation}
For any  control $u\in L^2_\cF(t, T;  \R^l)$, there exists
a unique  solution $X^{t,x_t,u}\in L^2_\cF(\Omega; \, C(t,  T;  \R^n))$.

At any time $t$ with the system state $X_t=x_t$, our aim  is to minimize
\begin{eqnarray}\label{costgeneral}
J(t,x_t;u)&\trieq&\frac{1}{2}\mathbb E_t\int_t^T[ \ang{Q_sX_s, X_s}+\ang{ R_su_s, u_s}]ds+\frac{1}{2}\mathbb E_t [\ang{G X_T, X_T}]\nonumber \\
&&- \frac{1}{2} \ang{h\Et{X_T}, \Et{X_T}}-\ang{ \mu_1 x_t+\mu_2,  \Et{X_T}}
\end{eqnarray}
over  $u\in L^2_\cF(t, \, T; \, \R^l)$,
where $X=X^{t,x_t,u}$, and $\Et{\cdot}=\E{\cdot|\cF_t}$.
Here $Q$ and $R$ are both given essentially bounded adapted processes on $[0,T]$ with values in ${\mathbb S}^n$
and ${\mathbb S}^l$ respectively,  $G, h, \mu_1, \mu_2$ are all constants in $\mathbb S^n$, $\mathbb S^n$,   $\R^{n\times n}$ and $\R^n$ respectively.
Throughout this paper, we assume that $Q\succeq 0$, $R\succeq 0$ a.s.,a.e., and $G\succeq 0$.

The first two terms in the cost  functional (\ref{costgeneral}) are standard in a classical LQ control problem, whereas the last two
are unconventional. Specifically, the term $- \frac{1}{2}\ang{h\mathbb E_t[ X_T], \Et{X_T}}$ is motivated by the variance term in a mean--variance
portfolio choice model \cite{HZ,ZL}, and the last term, $-\ang{\mu_1 x_t+\mu_2,  \Et{X_T}}$, which depends on the state $x_t$ at time $t$,  stems
from a state-dependent utility function in economics \cite{BMZ}.

Each of  these two terms introduces time-inconsistency of the underlying model in somewhat different ways. With the time-inconsistency, the notion
``optimality" needs to be defined in an appropriate way. Here we adopt the concept of equilibrium solution, which is, for any $t\in [0,T)$,  optimal
only for spike variation in an infinitesimal way.

Given a  control $u^*$. For any $t\in [0,T)$, $\varepsilon>0$ and $v\in L^2_{\cF_t}(\Omega; \, \R^l)$,  define
\begin{equation}\label{svgeneral}
u^{t,\varepsilon,v}_s=u^*_s+v\id_{s\in [t,t+\varepsilon)},\;\;\;s\in[t,T].
\end{equation}

\begin{definition}
Let $u^*\in L^2_\cF(0, \, T; \, \R^l)$ be a given control and $X^*$ be the state process corresponding to $u^*$. The control $u^*$ is called an equilibrium
 if
$$\lim_{\varepsilon\downarrow 0} \frac{J(t,X^*_t; u^{t,\varepsilon,v})-J(t,X^*_t;u^*)}{\varepsilon}\ge 0,$$
where $u^{t,\varepsilon,v}$ is defined by (\ref{svgeneral}), for any $t\in [0,T)$ and $v\in L^2_{\cF_t}(\Omega; \, \R^l)$.
\end{definition}

Notice that an equilibrium control here is defined in the class of {\it open-loop} controls, which is different from
the one in \cite{Basak}, \cite{BM}, \cite{BMZ},\cite{EL} and \cite{EP} where only {\it feedback} controls are considered. In our definition,
the perturbation of the control
in $[t, t+\varepsilon)$ will not change the control process in $[t+\varepsilon, T]$, whereas it is not the case with feedback controls.

In this paper, we will characterize equilibriums in general case and identify them in some special cases including that of the mean--variance
portfolio selection.

\section{Sufficient Condition of Equilibrium Controls}\label{formal-derivation}

In this section we present a general sufficient condition for equilibriums.
We derive this condition by  the second-order expansion in the spike variation, in the same spirit of proving the stochastic
Pontryagin's maximum principle \cite{peng90,YZ}.

Let $u^*$ be a fixed control and $X^*$ be the corresponding state process. For any $t\in [0, T)$,
define in the time interval $[t, T]$  the  processes $(p(\cdot;t), (k^j(\cdot;t))_{j=1,\cdots, d})\in L^2_\cF(t,T;\R^n)\times (L^2_\cF(t,T;\R^n))^d$ and
$(P(\cdot;t), (K^j(\cdot;t))_{j=1,\cdots,d})\in L^2_\cF(t,T;\mathbb S^n)\times (L^2_\cF(t,T;\mathbb S^n))^d$
as the solutions to the following equations:
\begin{eqnarray}
&&\left\{\begin{array}{ll}
dp(s;t)=&\hspace{-0.3cm}-[A_s'p(s;t)+\sum_{j=1}^d(C_s^j)' k^j(s;t)+Q_sX^*_s]ds\\
        &\hspace{-0.3cm} +\sum_{j=1}^dk^j(s;t)dW_s^j,\;\;s\in[t,T],\\
p(T;t)=&\hspace{-0.3cm}G X^*_T- h \Et{X^*_T}-\mu_1 X_t^*-\mu_2;
\end{array}\right.  \label{adjoint1general} \\
&&\left\{\begin{array}{ll}
dP(s;t)=&\hspace{-0.3cm}-\Big\{A_s'P(s;t)+P(s;t)A_s \\
&\hspace{-0.4cm}+\sum_{j=1}^d[(C_s^j)'P(s;t)C_s^j+(C_s^j)'K^j(s;t)+K^j(s;t)C_s^j]+Q_s\Big\}ds\\
&\hspace{-0.4cm}+\sum_{j=1}^dK^j(s;t)dW_s^j,\;\;s\in[t,T],\\
P(T;t)=&\hspace{-0.3cm}G.
\end{array}\right. \label{adjoint2general}
\end{eqnarray}

Note that for each fixed $t\in[0,T]$, the above equations are backward stochastic differential equations (BSDEs).
So these essentially form a flow of BSDEs. From the assumption that   $Q\succeq 0$ and $G\succeq 0$, it follows that
$P(s;t)\succeq 0$.

\begin{proposition}\label{variate}
 For any $t\in [0,T)$, $\varepsilon>0$ and $v\in L^2_{\cF_t}(\Omega; \, \R^l)$,  define $u^{t,\varepsilon,v}$
by (\ref{svgeneral}). Then
\begin{eqnarray}\label{epsilongeneral}
J(t,X^*_t; u^{t,\varepsilon,v})-J(t,X^*_t; u^*)
=\mathbb E_t\int_t^{t+\varepsilon} \left\{\ang{\Lambda(s;t),v}
+\frac{1}{2}\ang{H(s;t)v, v}\right\}ds+o(\varepsilon)
\end{eqnarray}
where $\Lambda(s;t)\trieq   B_sp(s;t)+\sum_{j=1}^d(D_s^j)'k^j(s;t)+R_su^*_s$ and
$H(s;t)\trieq R_s+\sum_{j=1}^d (D_s^j)'P(s;t)D_s^j$.
\end{proposition}

\begin{proof}
Let $X^{t,\varepsilon,v}$ be the state process corresponding to $u^{t,\varepsilon,v}$. Then
by the standard perturbation approach (see, e.g., \cite{YZ}), we have
\[ X^{t,\varepsilon,v}_s=X^*_s+Y^{t,\varepsilon,v}_s+Z^{t,\varepsilon,v}_s, \;\;s\in[t,T],\]
where $Y\equiv Y^{t,\varepsilon,v}$ and $Z\equiv Z^{t,\varepsilon,v}$ satisfy
\begin{eqnarray*}
\left\{
\begin{array}{l}
dY_s=A_sY_sds + \sum_{j=1}^d[C_s^jY_s+D_s^j v\id_{s\in [t, t+\varepsilon)}]dW_s^j,\;\;s\in[t,T],\\
Y_t=0;
\end{array}\right.\\
\left\{\begin{array}{l}
dZ_s=[A_sZ_s+B_s' v\id_{s\in [t, t+\varepsilon)}]ds + \sum_{j=1}^dC_s^jZ_sdW_s^j,\;\;s\in[t,T],\\
Z_t=0.
\end{array}\right.
\end{eqnarray*}
 Moreover
$$\Et{Y_s}=0,\;\; \Et{\sup_{s\in [t,T]}|Y_s|^2}=O(\varepsilon),\;\;
\Et{\sup_{s\in [t,T]}|Z_s|^2}=O(\varepsilon^2).$$

By these  estimates, we can calculate
\begin{eqnarray*}
&&2[J(t,X^*_t; u^{t,\varepsilon,v})-J(t,X^*_t,u^*)]\\
&=&\mathbb E_t\int_t^T\left[\ang{ Q_s (2X^*_s+Y_s+Z_s),  Y_s +Z_s }
+\ang{  R_s (2u^*_s+ v),  v}\id_{s\in [t, t+\varepsilon)}    \right]ds\\
&& +2\Et{\ang{G X^*_T,  Y_T+Z_T }}  +\Et{\ang{G (Y_T+Z_T), Y_T+Z_T}}\\
&& - 2\ang{h\Et{X^*_T}+\mu_1 X_t^*+\mu_2, \Et{Y_T+Z_T}}
 -\ang{h\Et{Y_T+Z_T}, \Et{Y_T+Z_T}}\\
&=&\mathbb E_t\int_t^T\left[\ang{ Q_s (2X^*_s+Y_s+Z_s),  Y_s +Z_s }
+\ang{  R_s (2u^*_s+ v),  v}\id_{s\in [t, t+\varepsilon)}    \right]ds\\
&& +\Et{2\ang{GX^*_T-h\mathbb E_t [X^*_T]-\mu_1 X_t^*-\mu_2, Y_T+Z_T }  +\ang{ G(Y_T+Z_T),Y_T+Z_T}}+o(\varepsilon).
\end{eqnarray*}

Recalling that  $(p(\cdot; t),k(\cdot; t))$ and $(P(\cdot; t),K(\cdot; t))$ solve respectively (\ref{adjoint1general}) and (\ref{adjoint2general}), we have
\begin{eqnarray*}
&&\Et {\ang{GX^*_T-h\Et{ X^*_T}-\mu_1 X_t^*-\mu_2, Y_T+Z_T} }\\
&=&\mathbb E_t\int_t^T\{ \ang{p(s;t), A_s(Y_s+Z_s)+B_s' v\id_{s\in [t, t+\varepsilon)} }\\
&&-\ang{ A_s' p(s;t)+\sum_{j=1}^d(C_s^j)' k^j(s;t)+Q_sX^*_s, Y_s+Z_s}\\
&&+\sum_{j=1}^d\ang{k^j(s;t), C_s^j(Y_s+Z_s)+D_s^j v \id_{s\in [t, t+\varepsilon)}}\}ds \\
&=&\mathbb E_t\int_t^T[\ang{-Q_sX^*_s,  Y_s+Z_s}+\ang{ B_sp(s;t)+\sum_{j=1}^d(D_s^j)'k^j(s;t),  v \id_{s\in [t, t+\varepsilon)}}]ds;
\end{eqnarray*}
and
\begin{eqnarray*}
&&\Et {\ang{G(Y_T+Z_T), Y_T+Z_T}}\\
&=&\mathbb E_t\int_0^T\left[-\ang{Q_s(Y_s+Z_s), Y_s+Z_s}+\sum_{j=1}^d \ang{(D_s^j)'P(s;t)D_s  v, v}\id_{s\in [t, t+\varepsilon)} \right]ds
+o(\varepsilon).
\end{eqnarray*}

This proves (\ref{epsilongeneral}).
\end{proof}

\medskip

It follows from  $R\succeq0$ and $P(s;t)\succeq 0$ that  $H(s;t)\succeq 0$.
In view of (\ref{epsilongeneral}),
a sufficient condition for an equilibrium is
\begin{equation}\label{Lambdacond}
\mathbb E_t\int_t^T|\Lambda(s;t)|ds<+\infty,\;\qquad \lim_{s\downarrow t} \Et{\Lambda(s;t)}=0,\;\as, \;\forall t\in[0,T].
\end{equation}

Under some condition, the second equality in (\ref{Lambdacond}) is ensured by
\begin{equation}\label{maximalprinciple}
R_t u^*_t+B_tp(t;t)+\sum_{j=1}^d(D_t^j)'k^j(t;t)=0,\;\as, \;\forall t\in[0,T].
\end{equation}

The following is the main general result for the time-inconsistent stochastic LQ control.

\begin{theorem}\label{maingeneral}
If the following system of  stochastic differential equations
\begin{equation}\label{fbsdegeneral}
\left\{
\begin{array}{l}
dX^*_s=[A_sX^*_s+B_s'u^*_s+b_s]ds+ \sum_{j=1}^d[C_s^jX^*_s+D_s^{j} u^*_s+\sigma_s^{j} ]dW_s^{j},\; s\in[0,T],\\
X^*_0=x_0,\\
dp(s;t)=-[A_s'p(s;t)+\sum_{j=1}^d(C_s^j)' k^j(s;t)+Q_sX^*_s]ds+\sum_{j=1}^dk^j(s;t)dW_s^j,\; s\in[t,T],\\
p(T;t)=G X^*_T- h \Et{X^*_T}-\mu_1 X_t^*-\mu_2
\end{array}\right.
\end{equation}
admits a solution $(u^*, X^*,p,k)$,  for any $t\in [0, T)$, such that  $\Lambda(\cdot; t)\trieq B_\cdot p(\cdot;t)+\sum_{j=1}^d(D_\cdot^j)'k(\cdot;t)^j+R_\cdot u^*_\cdot$
satisfies condition (\ref{Lambdacond}), and $u^*\in L^2_{\cF}(0,T; \R^l)$, then $u^*$ is an equilibrium control.
\end{theorem}

\begin{proof}
Given $(u^*, X^*, p, k)$ satisfying the conditions in this theorem,
at any time $t$, for any  $v\in L^2_{\cF_t}(\Omega; \R^l)$, define $\Lambda$ and $H$ as in Proposition \ref{variate}.  Then
\begin{eqnarray*}
\lim_{\varepsilon\downarrow 0} \frac{J(t,X^*_t; u^{t,\varepsilon})-J(t,X^*_t;u^*)}{\varepsilon}
&=&\lim_{\varepsilon\downarrow 0}\frac{\mathbb E_t\int_t^{t+\varepsilon}\left\{\ang{\Lambda(s;t), v}
+\frac{1}{2}\ang{H(s;t)v, v}\right\}ds}{\varepsilon}\\
&\geq &\lim_{\varepsilon\downarrow 0}\frac{\int_t^{t+\varepsilon}\ang{\Et{\Lambda(s;t)}, v}ds}{\varepsilon}\\
&\ge &0,
\end{eqnarray*}
proving the result.
\end{proof}

%

Theorem \ref{maingeneral} involves the existence of solutions to a flow of FBSDEs along with other conditions.
Proving the general existence remains an outstanding open problem. In the rest of this paper we will focus on the case when $n=1$. This case is important especially in
financial applications, as will be demonstrated by the mean--variance portfolio selection model.

When $n=1$, the state process $X$ is a  scalar-valued process evolving by the dynamics
\begin{equation}\label{control}
dX_s=[A_sX_s+B_s'u_s+b_s]ds+ [C_sX_s+D_s u_s+\sigma_s]'dW_s;\quad X_0=x_0,
\end{equation}
where $A$ is a bounded deterministic  scalar function on $[0, T]$.
The other parameters $B,C,D$ are all essentially bounded  and $\cF_t$-adapted processes on $[0,T]$
with values in $\R^{ l}$,
$\R^d$, $\R^{ d\times l}$,  respectively. Moreover,  $b\in L^2_\cF(0,T; \R)$ and $\sigma\in L^2_\cF(0,T; \R^d)$.

In this case,  the two adjoint equations for the equilibrium become
\begin{eqnarray}
&&\left\{\begin{array}{l}
dp(s;t)=-[A_sp(s;t)+C'_s k(s;t)+Q_sX^*_s]ds+k(s;t)'dW_s,\;\;s\in[t,T],\\
p(T;t)=G X^*_T- h \mathbb E_t[X^*_T]-\mu_1 X_t^*-\mu_2;
\end{array}\right.  \label{adjoint1} \\
&&\left\{\begin{array}{l}
dP(s;t)=-[(2A_s+|C_s|^2) P(s;t)+2C'_sK(s;t)+Q_s]ds\\
\quad\quad\quad\quad+K(s;t)'dW_s,\;\;s\in[t,T],\\
P(T;t)=G.
\end{array}\right. \label{adjoint2}
\end{eqnarray}

For reader's convenience,  we state here the $n=1$ version of Theorem \ref{maingeneral}:
\begin{theorem}\label{main}
If the following system of  stochastic differential equations
\begin{equation}\label{fbsde}
\left\{
\begin{array}{l}
dX^*_s=[A_sX^*_s+B'_su^*_s+b_s]ds+ [C_sX_s^*+D_su^*_s+\sigma_s]'dW_s,\;s\in[0,T],\\
X_0^*=x_0,\\
dp(s;t)=-[A_sp(s;t)+C'_sk(s;t)+Q_sX^*_s]ds+k(s;t)'dW_s,\;s\in [t,T],\\
p(T;t)=G X^*_T-h E_t[X^*_T]-\mu_1 X_t^*-\mu_2,\;t\in[0,T]
\end{array}\right.
\end{equation}
admits a solution $(u^*, X^*,p,k)$, for any $t\in [0, T)$, such that $\Lambda(\cdot; t)\trieq p(\cdot;t)B_{\cdot}+D'_{\cdot}k(\cdot;t)+R_{\cdot} u^*_{\cdot}$ satisfies the condition (\ref{Lambdacond}),
and  $u^*\in L^2_\cF(0,T;\R^l)$,
then $u^*$ is an equilibrium control.
\end{theorem}

\section{Equilibrium When Coefficients Are  Deterministic}\label{Deterministic-case}

Theorem \ref{main} shows that one can obtain  equilibrium  controls by solving the system of FBSDEs (\ref{fbsde}). However, the FBSDEs in (\ref{fbsde}) are not standard since a ``flow" of unknowns $(p(\cdot;t), k(\cdot;t))$ is involved. Moreover,  there is an additional
constraint (\ref{Lambdacond}), which under some condition boils down to an algebraic constraint (\ref{maximalprinciple})
that acts on the ``diagonal" (i.e. when $s=t$) of the flow. 
The unique solvability of this type of equations remains a challenging open problem even for the case $n=1$.
However, we are able to solve quite thoroughly this problem when the parameters
$A,B, C, D,b,\sigma, Q$ and $R$ are all deterministic functions.

Throughout this section we assume all the parameters are deterministic functions of $t$. In this case,
the BSDE (\ref{adjoint2}) turns out to be an ODE with solution
$K\equiv 0$ and
$P(s;t)=Ge^{\int_s^T(2A_u+|C_u|^2)du}+\int_s^Te^{\int_s^v(2A_u+|C_u|^2)du}Q_vdv$.

\subsection{An Ansatz}

As in the classical LQ control (see, e.g. \cite{YZ}), we attempt to look for a linear feedback equilibrium.
For this, given any $t\in [0,T]$, we consider the following
Ansatz:
\begin{equation}\label{pst}
p(s;t)=M_sX^*_s -N_s\mathbb E_t[X^*_s] -\Gamma^{(1)}_s X_t^*+\Phi_s,\;\;0\leq t\leq s\leq T,
\end{equation}
where  $M,N,\Gamma^{(1)},\Phi$ are deterministic differentiable functions with $\dot M=m, \dot N=n, \dot \Gamma^{(1)}=\gamma^{(1)}$ and $\dot \Phi=\phi$.

For any fixed $t$, applying Ito's formula to (\ref{pst}) in the time variable $s$ ,
we get
\begin{equation}\label{pst2}
\begin{array}{rl}
&dp(s;t)\\
=&\{M_s(A_sX^*_s+B_s'u^*_s+b_s)+m_sX^*_s-N_s \Et{A_sX^*_s+B_s'u^*_s+b_s}-n_s\Et{X^*_s}\\
&\ -\gamma_s^{(1)}X^*_t+\phi_s\}ds
+M_s(C_sX_s^*+D_su^*_s+\sigma_s)'dW_s.
\end{array}
\end{equation}
Comparing the
 $dW_s$ term with the $dW_s$  term  of $dp(s;t)$ in (\ref{fbsde}), we obtain
\begin{equation}\label{kst}
k(s;t)=M_s[C_sX^*_s+D_su^*_s+\sigma_s],\;\;s\in [t,T].
\end{equation}
Notice that $k(s;t)$ turns out to be independent of $t$.

Now we  ignore the difference between the conditions (\ref{Lambdacond}) and (\ref{maximalprinciple}), and put the above expressions of $p(s;t)$ and $k(s;t)$ into (\ref{maximalprinciple}). Then  we have
$$[(M_s-N_s-\Gamma^{(1)}_s) X^*_s+\Phi_s]B_s+M_sD'_s[C_sX^*_s+D_su^*_s+\sigma_s]+R_su^*_s=0,\;\;s\in [0,T],$$
from which we formally deduce
\begin{equation}\label{ustars}
u^*_s=\alpha_s X^*_s+\beta_s,
\end{equation}
where
\begin{eqnarray*}
\alpha_s&\trieq&  -(R_s+M_sD'_sD_s)^{-1}[(M_s-N_s-\Gamma_s^{(1)})B_s+M_sD'_sC_s],\\
\beta_s&\trieq& -(R_s+M_sD'_sD_s)^{-1}(\Phi_sB_s+M_sD'_s\sigma_s).
\end{eqnarray*}

Next, comparing the $ds$ term in (\ref{pst2}) with the one in (\ref{fbsde}) (we suppress the argument $s$ here), we obtain
\begin{eqnarray*}
0&=&m X^* +M(AX^*+B'u^*+b)-n \mathbb E_t[X^*]-N(A\mathbb E_t[X^*]+B'\mathbb E_t[u^*]+b)-\gamma^{(1)} X^*_t+\phi\\
&&+AMX^*-AN\mathbb E_t[X^*]- A\Gamma^{(1)} X_t^*+A\Phi+MC'[CX^*+Du^*+\sigma]+QX^*\\
&=&[m+2MA+M|C|^2+Q+(MB'+MC'D)\alpha]X^*-[n+2NA+NB'\alpha]\Et{X^*}\\
&&-(\gamma^{(1)}+A\Gamma^{(1)})X_t^*
+[(M-N)(B'\beta+b)+\phi+A\Phi+MC'(D\beta+\sigma)].
\end{eqnarray*}
Notice in the above $X^*\equiv X^*_s$ and $\Et{X^*}\equiv \Et{X^*_s}$ due to the omission of $s$.
This leads to the following equations for $M,N,\Gamma^{(1)},\Phi$ (again the argument $s$ is suppressed):
\begin{equation}\label{Ric1}
\left\{\begin{array}{l}
\dot{M}+(2A+|C|^2)M+Q\\
\;\;-M(B'+C'D)(R+MD'D)^{-1}[(M-N-\Gamma^{(1)})B+MD'C]=0,\;\;s\in[0,T],\\
M_T=G;
\end{array}\right.
\end{equation}
\begin{equation}\label{Ric2}
\left\{\begin{array}{l}
\dot{N}+2AN-NB'(R+MD'D)^{-1}[(M-N-\Gamma^{(1)})B+MD'C]=0,\;\;s\in[0,T],\\
N_T=h;
\end{array}\right.
\end{equation}
\begin{equation}\label{Ric3}
\left\{\begin{array}{l}
  \dot{\Gamma}^{(1)}=-A\Gamma^{(1)},\;\;s\in[0,T],\\
  \Gamma^{(1)}_T=\mu_1 ;
  \end{array}\right.
\end{equation}
\begin{equation}\label{Ric4}
\left\{\begin{array}{l}
\dot{\Phi}+\{A-[(M-N) B'+MC'D](R+MD'D)^{-1}  B\}\Phi+(M-N)b+C'M\sigma\\
\;\; -[(M-N) B'+MC'D](R+MD'D)^{-1}MD'\sigma =0,\;\;s\in[0,T],\\
\Phi_T=-\mu_2.
\end{array}\right.
\end{equation}

The solution to equation (\ref{Ric3}) is $\Gamma^{(1)}_s=\mu_1 e^{\int_s^TA_tdt}$.
Equations (\ref{Ric1}) and (\ref{Ric2}) form a system of coupled Riccati equations\footnote{Strictly speaking, these are not
Riccati equations in the usual sense as they are not symmetric. However, we still use the term so as to see the connection and difference between
time-inconsistent and time-consistent LQ control problems.} for $(M,N)$
\begin{equation}\label{systemRic}
\left\{\begin{array}{lll}
\dot{M}\hspace{-0.3cm}&=&\hspace{-0.3cm}- \left[2A+|C|^2+\Gamma^{(1)} B'(R+MD'D)^{-1}(B+D'C)\right] M-Q\\
  &&\hspace{-0.3cm}+(B+D'C)'(R+MD'D)^{-1}(B+D'C)M^2-B'(R+MD'D)^{-1}(B+D'C)MN,\\
M_T\hspace{-0.3cm}&=&\hspace{-0.3cm} G;\\
\dot{ N}\hspace{-0.3cm}&=&\hspace{-0.3cm}- \left[2A+\Gamma^{(1)} B'(R+MD'D)^{-1}B\right] N+B'(R+MD'D)^{-1}(B+D'C)MN\\
   && -B'(R+MD'D)^{-1}BN^2,\\
 N_T\hspace{-0.3cm}&=&\hspace{-0.3cm}h.\\
\end{array}\right.
\end{equation}

Finally, once we get  the solution for $(M,N)$, equation (\ref{Ric4}) is a simple ODE.
Therefore, it is crucial to solve (\ref{systemRic}), which will be carried out in the next subsection.

\subsection{Solution to Riccati System (\ref{systemRic})}
Formally, we define $J=\frac{M}{N}$, and  study the following equation for $(M,J)$:
\begin{equation}\label{MJformal}
\left\{\begin{array}{lll}
\dot{M}\hspace{-0.3cm}&=&\hspace{-0.3cm}- \left[2A+|C|^2+\Gamma^{(1)} B'(R+MD'D)^{-1}(B+D'C)\right] M-Q\\
&&\hspace{-0.3cm}+(B+D'C)'(R+MD'D)^{-1}(B+D'C)M^2-B'(R+MD'D)^{-1}(B+D'C)\frac{M^2}{J},\\
M_T\hspace{-0.3cm}&=&\hspace{-0.3cm}G;\\
\dot{J}\hspace{-0.3cm}&=&\hspace{-0.3cm}-[|C|^2- C'D(R+MD'D)^{-1}(B+D'C)M+ \Gamma^{(1)} B'(R+MD'D)^{-1}D'C+\frac{Q}{M}] J \\
   &&\hspace{-0.3cm}-B'(R+MD'D)^{-1}D'C M,\\
 J_T\hspace{-0.3cm}&=&\hspace{-0.3cm}\frac{G}{h}.\\
\end{array}\right.
\end{equation}

\begin{proposition}\label{MJequivMN}
If the system (\ref{MJformal}) admits a positive solution pair $(M, J)$, then the system (\ref{systemRic})
admits a positive solution pair $(M, \frac{M}{J})$.
\end{proposition}

\begin{proof}
The proof is straightforward.
\end{proof}

In the following two subsections, we will study the system (\ref{MJformal}) for two cases respectively. The main technique
is the truncation method. This method involves  ``truncation functions"
$\cdot \vee c$ for a small number $c>0$,  and $\cdot\wedge K$ for a large number $K$.

\subsubsection{Standard case}
We first consider the standard case where 
$R-\delta I\succeq 0$ for some $\delta>0$.

\begin{theorem}\label{exisMJformal}
Assume that $R-\delta I\succeq 0$ for some $\delta>0$ and $G\ge h>0$.
Then  
(\ref{MJformal}) and (\ref{systemRic}) admit  unique positive solution pairs
 if $\frac{QD'D+|C|^2R}{l}+\Gamma^{(1)}{\cS}(D'CB' )\succeq 0$, and
 either (i) there exists a constant $\lambda\ge 0$ such that $B=\lambda D'C$,  or (ii) $D'D-\delta I\succeq 0$ for some $\delta>0$.
\end{theorem}

\begin{proof} For fixed $c>0$ and $K>0$, consider the following truncated system of (\ref{MJformal}):
\begin{equation}\label{MJformaltr}
\left\{\begin{array}{lll}
\dot{M}\hspace{-0.3cm}&=&- \left[2A+|C|^2+\Gamma^{(1)} B'(R+M^+D'D)^{-1}(B+D'C)\right] M-Q\\
&&+(B+D'C)'(R+M^+D'D)^{-1}(B+D'C)M(M^+\wedge K)\\
&&-B'(R+M^+D'D)^{-1}(B+D'C)\frac{M(M^+\wedge K)}{J\vee c},\\
M_T\hspace{-0.3cm}&=&G;\\
\dot{J}\hspace{-0.3cm}&=&-\lambda^{(1)} J -B'(R+M^+D'D)^{-1}D'C (M^+\wedge K),\\
 J_T\hspace{-0.3cm}&=&\frac{G}{h}\\
\end{array}\right.
\end{equation}
where $M^+=\max\{M, 0\}$ and
$$\lambda^{(1)}\trieq |C|^2- C'D(R+M^+D'D)^{-1}(B+D'C)(M^+\wedge K)+ \Gamma^{(1)} B'(R+M^+D'D)^{-1}D'C+\frac{Q}{M\vee c}.$$

Since $R-\delta I\succeq 0$, the above system is locally Lipschitz with linear growth,
hence it admits a unique solution $(M^{c,K}, J^{c,K})$. We omit the superscript
$(c, K)$ when no confusion might arise.

We are going to prove that $J\ge 1$, and $M\in [\eta, L]$ for some $\eta>0$ and $L>0$
independent of $c$ and $K$ appearing in the truncation functions.
To this end, denote
\begin{eqnarray*}
\lambda^{(2)}&=& (2A+|C|^2+\Gamma^{(1)} B'(R+M^+D'D)^{-1}(B+D'C))\\
&&-(B+D'C)'(R+M^+D'D)^{-1}(B+D'C)(M^+\wedge K)\\
&&   +B'(R+M^+D'D)^{-1}(B+D'C)\frac{M^+\wedge K}{J\vee c}.
\end{eqnarray*}
 Then $\lambda^{(2)}$ is bounded, and $M$ satisfies
\begin{equation} \label{linearM}
  \dot M+\lambda^{(2)} M+Q=0,\; M_T=G.
\end{equation}
Hence $M>0$. As a result, the terms $R+M^+D'D$ and $M^+$ can be
replaced by $R+MD'D$ and $M$ respectively in (\ref{MJformaltr}) without changing their values.

Now we prove $J\ge 1$. Denote $\tilde J\trieq J-1$, then $\tilde J$ satisfies the ODE
\begin{eqnarray*}
\dot{\tilde J}&=&-\lambda^{(1)}\tilde J-\left[ \lambda^{(1)}+B'(R+MD'D)^{-1}D'C(M\wedge K)\right] \\
&=&-\lambda^{(1)}\tilde J-a^{(1)}
\end{eqnarray*}
where
\begin{eqnarray*}
a^{(1)}&=&\lambda^{(1)}+B'(R+MD'D)^{-1}D'C(M\wedge K)\\
&=&|C|^2-C'D(R+MD'D)^{-1}D'C(M\wedge K)+ \Gamma^{(1)}B'(R+MD'D)^{-1}D'C+\frac{Q}{M\vee c}\\
&\ge &|C|^2-C'D(R+MD'D)^{-1}D'C)M+ \Gamma^{(1)}B'(R+MD'D)^{-1}D'C+\frac{Q}{M\vee c}\\
&=&\tr \left\{(R+MD'D)^{-1}\frac{|C|^2+Q/(M\vee c)}{l}(R+MD'D)\right\}\\
&&      -\tr\{(R+MD'D)^{-1}D'CC'DM\}+\tr\{(R+MD'D)^{-1}\Gamma^{(1)}D'CB'\}\\
 &=&\tr\left\{(R+MD'D)^{-1}H \right\}
\end{eqnarray*}
with $H\trieq\frac{|C|^2+Q/(M\vee c)}{l}(R+MD'D)-D'CC'DM+\Gamma^{(1)}\cS(D'CB' ) $.

When $c$ is small enough such that 
$R-cD'D\succeq 0$, we have
$$\frac{Q}{M\vee c}(R+MD'D)\ge QD'D.$$
Furthermore,
$$\frac{|C|^2}{l}D'D-D'CC'D\succeq 0.$$
Hence,
\begin{eqnarray*}
H\succeq \frac{QD'D+|C|^2R}{l}+\Gamma^{(1)}\cS(D'CB')\succeq 0,
\end{eqnarray*}
and consequently $a^{(1)}\ge \tr\{(R+MD'D)^{-1}H\} \ge 0$.\footnote{Here we used the inequality that $\tr(AB)\ge 0$ for any  positive semi-definite matrices $A,B$.} We deduce that $\tilde J\ge 0$,
or equivalently $J\ge 1$.

Next we prove $M$ is bounded above by a constant $L>0$ independent of the truncation.
Choosing $c$ small enough,  the equation for $M$ turns out to be
$$ \left\{\begin{array}{l}
-\dot{M}= \left(2A+|C|^2+\Gamma^{(1)} B'(R+MD'D)^{-1}(B+D'C)\right)M+Q-k M(M\wedge K),\\
M_T=G
\end{array}\right.
$$
where
\begin{eqnarray*}
k&=&(B+D'C)'(R+MD'D)^{-1}(B+D'C)-B'(R+MD'D)^{-1}(B+D'C)\frac{1}{J}\\
&=&B'(R+MD'D)^{-1}B\left(1-\frac{1}{J}\right)+B'(R+MD'D)^{-1}D'C\left(2-\frac{1}{J}\right)\\
&&+C'D(R+MD'D)^{-1}D'C\\
&\ge&B'(R+MD'D)^{-1}D'C\left(2-\frac{1}{J}\right).
\end{eqnarray*}

If $B=\lambda D'C$ for some $\lambda\ge 0$, then we have $k\ge 0$. Hence $M$ admits an upper bound $L$  independent of $c$ and $K$.

If $D'D- \delta I\succeq 0$,
then $|kM|$ admits a bound independent of $c$ and $K$; hence once again $M$ admits an upper bound $L$ independent of $c$ and $K$.

Choosing $K=L$ and examining again equation (\ref{linearM}) we deduce that there exists $\eta>0$
independent of $c$ such that $M\ge \eta$. It now suffices to take $c=\eta$ to finish the proof.
\end{proof}

\subsubsection{Singular case}
Let us now consider the singular case $R\equiv0$. We suppose here that $D'D-\delta I\succeq 0$ for some $\delta>0$ in this
subsection. Then the system of $(M, J)$ is
 \begin{equation}\label{MJequation}
\left\{\begin{array}{lll}
\dot{M}\hspace{-0.3cm}&=&- \left[2A+|C|^2 -(B+D'C)'(D'D)^{-1}(B+D'C) +B'(D'D)^{-1}(B+D'C)\frac{1}{J}\right] M\\
&&-Q-\Gamma^{(1)} B'(D'D)^{-1}(B+D'C)\\
M_T\hspace{-0.3cm}&=&G;\\
\dot{J}\hspace{-0.3cm}&=&-[|C|^2- C'D(D'D)^{-1}(B+D'C)+ (\Gamma^{(1)} B'(D'D)^{-1}D'C+Q)\frac{1}{M}] J \\
   &&-B'(D'D)^{-1}D'C,\\
 J_T\hspace{-0.3cm}&=&\frac{G}{h}.
\end{array}\right.
\end{equation}
This system is even easier than the previous one.
We will use the same truncation argument  to prove the existence of a solution.

\begin{theorem}
  Given $G\ge h>0$, $R\equiv 0$ and $D'D-\delta I\succeq 0$ for some $\delta>0$.
  If $Q+\Gamma^{(1)} B'(D'D)^{-1}(B+D'C)\ge 0$ and $Q+\Gamma^{(1)} B'(D'D)^{-1}D'C\ge 0$, then  (\ref{MJequation}) and (\ref{systemRic}) admit positive
  solution pairs.
\end{theorem}

\begin{proof}
 For a fixed $c>0$, consider the following truncated system:
 \begin{equation}\label{MJequationtr}
\left\{\begin{array}{lll}
\dot{M}\hspace{-0.3cm}&=&- \left[2A+|C|^2 -(B+D'C)'(D'D)^{-1}(B+D'C)+B'(D'D)^{-1}(B+D'C)\frac{1}{J\vee c} \right] M\\
&&-Q-\Gamma^{(1)} B'(D'D)^{-1}(B+D'C),\\
M_T\hspace{-0.3cm}&=&G;\\
\dot{J}\hspace{-0.3cm}&=&-[|C|^2- C'D(D'D)^{-1}(B+D'C)+ (\Gamma^{(1)} B'(D'D)^{-1}D'C+Q)\frac{1}{M\vee c}] J \\
   &&-B'(D'D)^{-1}D'C,\\
 J_T\hspace{-0.3cm}&=&\frac{G}{h}.\\
\end{array}\right.
\end{equation}
This system is locally Lipschitz with linear growth, hence it admits a unique solution pair  $(M,J)$ depending on $c$.

Define $\tilde J=J-1$. Then
$$\dot{\tilde J}=-\lambda^{(3)} \tilde J-a^{(3)}$$
with $\lambda^{(3)}=|C|^2- C'D(D'D)^{-1}(B+D'C)+ (\Gamma^{(1)} B'(D'D)^{-1}D'C+Q)\frac{1}{M\vee c}$ being bounded, and
\begin{eqnarray*}
   a ^{(3)}&=&\lambda^{(3)}+B'(D'D)^{-1}D'C\\
   &=&|C|^2- C'D(D'D)^{-1}D'C+ (\Gamma^{(1)} B'(D'D)^{-1}D'C+Q)\frac{1}{M\vee c}\\
   &\ge &(\Gamma^{(1)} B'(D'D)^{-1}D'C+Q)\frac{1}{M\vee c}\\
   &\ge& 0.
\end{eqnarray*}
Hence $J\ge 1$. Now we choose $c\le 1$.

Denote $\lambda^{(4)}=2A+|C|^2 -(B+D'C)'(D'D)^{-1}(B+D'C)+B'(D'D)^{-1}(B+D'C)\frac{1}{J\vee c}$, $\tilde Q=Q+\Gamma^{(1)} B'(D'D)^{-1}(B+D'C)\ge 0$. Then
$|\lambda^{(4)}|$  admits a bound independent of $c$, and
$$\dot{M}+\lambda^{(4)} M+\tilde Q=0, \; M_T=G.$$
Hence there exists some $\eta>0$ (independent of $c$) such that $M\ge \eta$.
Choosing $c=\eta$, we conclude the proof.
\end{proof}

\subsection{Equilibrium Controls}

We now present the main result of this section.
\begin{theorem}\label{detmain}
Suppose $G\ge h>0$, The system of the Riccati equations (\ref{systemRic}) admits a unique positive solution pair $(M,N)$ in the following three cases:
\begin{itemize}
\item [(i)]$R-\delta I\succeq 0$ for some $\delta>0$,  $\frac{QD'D+|C|^2R}{l}+\Gamma^{(1)}\cS(D'CB') \succeq 0$ and $B=\lambda D'C$ for some $\lambda\ge 0$;
\item [(ii)]  $R-\delta I\succeq 0$ for some $\delta>0$, $\frac{QD'D+|C|^2R}{l}+\Gamma^{(1)}\cS(D'CB') \succeq 0$ and $D'D-\delta I\succeq 0$ for some $\delta>0$;
\item [(iii)] $R\equiv 0$, $D'D-\delta I\succeq 0$ for some $\delta>0$, $Q+\Gamma^{(1)} B'(D'D)^{-1}(B+D'C)\ge 0, \, Q+\Gamma^{(1)} B'(D'D)^{-1}D'C\ge 0$.
\end{itemize}
Moreover,
let $\Phi$ be
a solution of ODE (\ref{Ric4}). Then  $u^*(\cdot)$  given by (\ref{ustars}) is an equilibrium.
\end{theorem}

\begin{proof}
Define $p(\cdot;\cdot)$ and $ k(\cdot;\cdot)$
by (\ref{pst}) and (\ref{kst}) respectively. It is straightforward  to check that $(u^*_\cdot, X^*_\cdot, p(\cdot;\cdot), k(\cdot;\cdot))$ satisfies the system of SDEs (\ref{fbsde}).

In all the three cases,  we can  check that
$\alpha_s$ and $\beta_s$ in (\ref{ustars}) are both uniformly bounded, hence $u^*\in L_\cF^2(0,T; \R^l)$
and $X^*_s\in L^2(\Omega;\, C(0,\, T; \, \R))$.

Finally, denote $\Lambda(s;t)=R_su^*_s+p(s;t)B+D_s'k(s;t)$. By plug $p, k, u^*$ defined in (\ref{pst}), (\ref{kst}) and (\ref{ustars}) into $\Lambda$,  we have
\begin{eqnarray*}
\Lambda(s;t)
&=&R_su^*_s+(M_sX^*_s -N_s\mathbb E_t[X^*_s] -\Gamma^{(1)}_s X_t^*+\Phi_s)B_s
+M_sD_s'[C_sX^*_s+D_su^*_s+\sigma_s]\\
&=&(R_s+M_sD_s'D_s)u^*_s+(B_s+D_s'C_s)M_sX^*_s
       -N_s\Et{X^*_s}B_s-\Gamma^{(1)}_sX^*_tB_s\\
       &&+(\Phi_sB_s+M_sD_s'\sigma_s)\\
 &=&-[(M_s-N_s-\Gamma^{(1)}_s)B_s+M_sD_s'C_s]X^*_s-\Phi_sB_s-M_sD_s'\sigma_s\\
&&+(B_s+D_s'C_s)M_sX^*_s-N_s\Et{X^*_s}B_s-\Gamma^{(1)}_sX^*_tB_s+(\Phi_sB_s+M_sD_s'\sigma_s)\\
 &=&(N_s+\Gamma^{(1)}_s)X^*_sB_s-N_s\Et{X^*_s}B_s-\Gamma^{(1)}_sX^*_tB_s\\
&=&N_s[X^*_s-\Et{X^*_s}]B_s+\Gamma^{(1)}_s(X^*_s-X^*_t)B_s.
\end{eqnarray*}
Clearly $\Lambda$ satisfies the first condition in (\ref{Lambdacond}). Furthermore,
we have
 $$\lim_{s\downarrow t} \Et{|X^*_s-\Et{X^*_s}|}=0,\qquad \mbox{ and }
 \lim_{s\downarrow t} \Et{|X^*_s-X^*_t|}=0;$$
  hence $\Lambda$ satisfies the second condition in (\ref{Lambdacond}).

By Theorem \ref{main}, $u^*$ is an equilibrium.
\end{proof}

\begin{remark}
{\rm If  $\mu_1\ge 0$ (e.g. in the mean--variance model to be studied subsequently), then  $\Gamma^{(1)}_t=\mu_1 e^{\int_t^TA_sds}\ge 0$.
With this condition, the first case and the third case in Theorem \ref{detmain} can be simplified as
\begin{itemize}
\item [(i')]$R-\delta I\succeq 0$ for some $\delta>0$, and $B=\lambda D'C$ for some $\lambda\ge 0$;
\item [(iii')] $R\equiv 0$, $D'D-\delta I\succeq 0$ for some $\delta>0$, and $ Q+\Gamma^{(1)} B'(D'D)^{-1}D'C\ge 0$.
\end{itemize}}

\end{remark}

\section{Mean-Variance Equilibrium Strategies in Complete Market}\label{random-case}

In this section, we study the continuous-time Markowitz's mean--variance portfolio selection model in a complete market.
The problem is inherently time inconsistent due to the variance term. Moreover, as in \cite{BMZ} we consider a state-dependent
mean expectation. Hence there are two different sources of time inconsistency. The definition of equilibrium strategies
is in the sense of open-loop, which is different from the feedback one in \cite{BM,BMZ}.

The model is mathematically a special case of the general LQ problem formulated earlier in this paper, with $n=1$ naturally. However,
some coefficients are allowed to be random; so it is not a direct application of the previous section. Indeed the analysis in
this section is much more involved due to the randomness of the coefficients.

For each $t\in [0, T)$, consider a wealth-portfolio
process $(X_t, \pi_t)$ satisfying the wealth equation
\begin{equation}\label{wealth}
 \left\{\begin{array}{l}
         dX_s=r_s X_sds+(\mu_s-r_s\id)'\pi_sds+\pi_s'\sigma_s dW_s,\qquad s\in [t,T],\\
     X_t=x_t,
        \end{array}
 \right.
\end{equation}
where $r\in  L_\cF^{\infty}(0,T; \R)$ is the interest rate process,
$\mu\in L_\cF^{\infty}(0,T; \R^d)$ and $ \sigma \in L_\cF^{\infty}(0,T; \R^{d\times d})$ are the drift rate vector
and volatility processes of risky assets respectively. We assume throughout that  $\sigma_s\sigma_s'-\varepsilon I\succeq 0$
for some $\varepsilon>0$ to ensure the completeness of the market .

Denote $\theta_t=\sigma_t^{-1}(\mu_t-r_t\id), u_t=\sigma_t'\pi_t$. Then
the wealth equation is equivalent to the equation of $(X_t ,u_t)$
\begin{equation}\label{wealth2}
 \left\{\begin{array}{l}
         dX_s=r_s X_sds+\theta_s'u_sds+u_s'dW_s,\qquad s\in [t,T],\\
     X_t=x_t.
        \end{array}
 \right.
\end{equation}

We interchangeably call $\pi$ and $u$ as (trading) strategies. It follows from our assumptions on $\theta$ that 
$\pi\in L_\cF^{2}(0,T; \R)$ if and only if $u\in L_\cF^{2}(0,T; \R)$. 
The objective of a mean-variance portfolio choice model  at time $t\in [0, T)$ is to achieve a balance between
conditional variance and conditional expectation of terminal wealth; namely, to choose a strategy $u$ so as to minimize
\begin{eqnarray}\label{mv-obj}
J(t, x_t; u)&\trieq &\frac{1}{2}Var_t(X_T)-(\mu_1 x_t+\mu_2) \mathbb E_t[X_T]\\
&=&\frac{1}{2}\left(\mathbb E_t[X_T^2]-(\mathbb E_t[X_T])^2\right)-(\mu_1 x_t+\mu_2) \mathbb E_t[X_T]\nonumber
\end{eqnarray}
with $\mu_1\ge 0$.
Here we insist that the weight between the conditional variance (as a risk measure) and the
conditional expectation should depend on the current wealth level, the reason having been elaborated in \cite{BMZ}.

When  the market parameters $r$ and $\theta$ are both deterministic, the problem is a special case of the one studied in
Section \ref{Deterministic-case}. In this section, we will find the equilibrium strategies for the model
where the  interest rate $r$ is deterministic but $\theta$ is  allowed to be random.

The problem (\ref{wealth}) --  (\ref{mv-obj}) is clearly a special case of LQ problem  (2.2) -- (2.3) with $n=1$.
The FBSDE (\ref{fbsde}) specializes  to
\begin{equation}\label{fbsderand}
\left\{
\begin{array}{l}
dX^*_s=[r_s X^*_s+\theta_s' u^*_s]ds+ (u^*_s)'dW_s,\quad X_0^*=x_0,\\
dp(s;t)=-r_s p(s;t)ds+k(s;t)'dW_s,\\
p(T;t)=X^*_T-\mathbb E_t[X^*_T]-\mu_1X^*_t-\mu_2,
\end{array}\right.
\end{equation}
and the process $\Lambda(s;t)$ in condition (\ref{Lambdacond}) is
$$\Lambda(s;t)=p(s;t)\theta_s+k(s;t).$$

\subsection{Formal Derivation}
As before, let us look for a solution in the form
\begin{equation}\label{prand}
p(s;t)=M_s X_s^*- \Gamma_s^{(1)} X_t^*+\Gamma_s^{(2)} - \mathbb E_t[N_s X_s^*+\Gamma_s^{(3)}] ,
\end{equation}
where $(M,U)$, $(N,V)$, $(\Gamma^{(1)}, \gamma^{(1)})$,  $(\Gamma^{(2)}, \gamma^{(2)})$ and $(\Gamma^{(3)}, \gamma^{(3)})$ are solutions of the following BSDEs:
\begin{equation}\label{5bsde}
\left\{
\begin{array}{rcl}
dM_s&=&-F_{M,U}ds+U_s 'dW_s,\quad M_T=1;\\
dN_s&=&-F_{N,V}ds+V_s' dW_s,\quad N_T=1;\\
d\Gamma^{(1)}_s&=&-F^{(1)}ds+(\gamma^{(1)}_s)' dW_s,\quad \Gamma^{(1)}_T=\mu_1;\\
d\Gamma^{(2)}_s&=&-F^{(2)}ds+(\gamma^{(2)}_s)'dW_s,\quad \Gamma^{(2)}_T=-\mu_2;\\
d\Gamma^{(3)}_s&=&-F^{(3)}ds+(\gamma^{(3)}_s)'dW_s,\quad \Gamma^{(3)}_T=0.
\end{array}\right.
\end{equation}

It is an easy exercise to obtain
\begin{eqnarray*}
d[N_sX^*_s]&=&[rNX^*+N\theta'u^*-X^*F_{N,V}+V'u^*]ds+[Nu^*+X^*V]'dW_s,\\
d \mathbb E_t[N_sX^*_s]&=&\mathbb E_t[rNX^*+N\theta'u^*-X^*F_{N,V}+V'u^*]ds,\\
d[M_sX^*_s]&=&[rMX^*+M\theta' u^*-X^*F_{M,U}+U'u^*]ds+[Mu^*+X^*U]'dW_s.
\end{eqnarray*}

Applying Ito's formula to $p(s;t)=M_s X_s^* + \Gamma_s^{(2)} - \mathbb E_t[N_s X_s+\Gamma_s^{(3)}] -\Gamma^{(1)}_s X^*_t$ and comparing the $dW_s$
term in the second equation of (\ref{fbsderand}), we get
\begin{equation}\label{krand}
k(s;t)=X_s^*U_s+M_s u_s^* +\gamma^{(2)}_s-\gamma^{(1)}_sX^*_t.
\end{equation}
Putting the expressions of $p$ and $k$ into the formal
condition $\Lambda(s;s)=0$, we obtain
\begin{eqnarray*}
 u^*_s&=&-M_s^{-1}\left[\left(\theta_s(M_s-N_s-\Gamma^{(1)}_s)+U_s-\gamma^{(1)}_s\right)X^*_s+\theta_s
 (\Gamma^{(2)}_s-\Gamma^{(3)}_s)+\gamma^{(2)}_s\right]\\
&=&\alpha_s X_s^* + \beta_s,
\end{eqnarray*}
where
$$\alpha_s\trieq -M_s^{-1}\left(\theta_s (M_s-N_s-\Gamma^{(1)}_s)+U_s-\gamma^{(1)}_s)\right), \quad \beta_s\trieq -M_s^{-1}\left(\theta_s (\Gamma^{(2)}_s-\Gamma^{(3)}_s)
+\gamma^{(2)}_s\right).$$

Applying again Ito's formula to $p$ and using the above expression of $u$, we deduce
\begin{eqnarray*}
 dp(s;t)&=& [-F_{M,U} X^*_s +r_s M_s X^*_s+(\theta_s  M_s +U_s ) (\alpha X^*_s + \beta_s)-F^{(2)}+X^*_t F^{(1)}]ds\\
& & +\mathbb E_t[F_{N,V} X^*_s -r_s N_s X^*_s-(\theta_s  N_s +V_s ) (\alpha X_s + \beta_s)+F^{(3)}]ds+k(s,t)'dW_s,
\end{eqnarray*}
while the second equation in (\ref{fbsderand}) gives
$$dp(s;t)=\{ -r_sM_sX^*_s+r_s\Gamma^{(1)}_sX^*_t -r_s\Gamma^{(2)}_s+r_s\mathbb E_t[N_sX^*_s+\Gamma^{(3)}_s] \}ds+k(s;t)'dW_s.$$
Comparing the corresponding terms, we obtain (again we supress the subscripts $s\in [t,T]$):
\begin{eqnarray*}\label{5bsdecoef}
 F_{M,U}&=& 2rM+(\theta M+U)'\alpha;\\
F_{N,V}&=& 2rN+(\theta N+V)'\alpha;\\
F^{(1)}&=&r\Gamma^{(1)};\\
F^{(2)}&=&r\Gamma^{(2)}+(\theta M+U)'\beta;\\
F^{(3)}&=&r\Gamma^{(3)} + (\theta N+V)'\beta.
\end{eqnarray*}

\subsection{Solution to the BSDEs (\ref{5bsde}) }

It now suffices to solve the  BSDEs (\ref{5bsde}). Its third equation  can be easily solved, whose solution is
$$\Gamma^{(1)}_t=\mu_1e^{\int_t^Tr_sds}, \; \gamma^{(1)}_t=0.$$
Noting that the first two equations are identical,
 we conclude that
$$M=N, \quad U=V.$$
Then
$$F^{(2)}-F^{(3)}=r(\Gamma^{(2)}-\Gamma^{(3)}).$$
By the last two equations in (\ref{5bsde}), we have
$$\Gamma_s^{(2)}-\Gamma_s^{(3)}=-\mu_2 e^{\int_s^T r_t dt}\trieq \Gamma_s.$$

To proceed, let us recall some facts about BMO martingales; see Kazamaki \cite{Kazamaki}.
The process $Z\cdot W\trieq\int_0^\cdot Z_s'dW_s$ is a BMO martingale if and only if there exists a constant $C>0$ such that
$$\mathbb E\left[\int_\tau^T |Z_s|^2ds\Big| {\cal F}_\tau \right] \le C$$
for all stopping times $\tau\le T$. For every such $Z$, the stochastic exponential of $Z\cdot W$ denoted by ${\cal E}(Z\cdot W)$ is a positive martingale;
and for any $p>1$, there exists a constant $C_p>0$ such that
$\E{\left(\int_\tau^T|Z_s|^2ds\right)^p\Big|\cF_\tau}\le C_p$ for any stopping time $\tau\le T$.
Moreover, if $Z\cdot W$ and $V\cdot W$ are both BMO martingales, then under the probability measure $\Q$
defined by $\frac{d\Q}{d\p}={\cal E}_T(V\cdot W)$, $W^\Q_t\trieq W_t-\int_0^tV_sds$ is a standard Brownian motion, and $Z\cdot W^\Q$ is a BMO martingale.

Now plug  the definition of $\alpha$ into the first equation in (\ref{5bsde}), we get the BSDE satisfied by $(M,U)$:
\begin{equation}\label{MU}
\left\{\begin{array}{l}
        dM_s=-(2r_s M_s-U_s'\theta_s +\Gamma^{(1)}_s|\theta_s|^2- M_s^{-1}|U_s|^2+\Gamma^{(1)}_s M_s^{-1}U_s'\theta_s)ds+U_s' dW_s,\\
M_T=1.\end{array}\right.
\end{equation}

This is a type of indefinite stochastic Riccati equation due to the presence of $M^{-1}$ in the driver; however it is different from the one studied in
\cite{HuZhou}.

\begin{proposition}
BSDE (\ref{MU}) admits a  unique solution $(M,U)\in L_{\cal F}^\infty(0,T;\mathbb R)\times L_{\cal F}^2(0,T;\mathbb R^d)$ satisfying  $M\ge c$ for some  constant $c>0$. Moreover, $U\cdot W$ is a BMO martingale.
\end{proposition}
\begin{proof}
 Once again, we will prove the existence by a truncation argument.
Let $c>0$ be a given number to be chosen later. Consider the following quadratic BSDE:
\begin{equation}\label{MUtrun}
\left\{\begin{array}{l}
        dM_s=-\left[2r_s M_s- U_s'\theta_s+\Gamma^{(1)}_s|\theta_s|^2- \frac{|U_s|^2}{M_s\vee c} +\Gamma^{(1)}_s \frac{U_s'\theta}{M_s\vee c} \right] ds+U_s' dW_s,\\
M_T=1.\end{array}\right.
\end{equation}

This BSDE is a standard quadratic BSDE. Hence there exists a solution $(M^{c},U^{c})\in L_{\cF}^\infty(0,T;\R)\times L_{\cF}^2(0,T;\R^d)$
and $U^{c}\cdot W$ is a BMO martingale; see \cite{Koby} and \cite{Morlais}.

We can rewrite the above BSDE as:
\begin{equation}\label{MUtrun2}
\left\{\begin{array}{l}
        dM_s=-(2r_s M_s+\Gamma^{(1)}_s|\theta_s|^2)ds+U_s' [dW_s-(\Gamma^{(1)}_s \frac{1}{M_s\vee c}\theta_s-\theta_s- \frac{1}{M_s\vee c}U_s )ds],\\
M_T=1.\end{array}\right.
\end{equation}

As $(\Gamma^{(1)} \frac{1}{M^c\vee c}\theta-\theta- \frac{1}{M^c\vee c}U^c)\cdot W$ is a BMO martingale, there exists a new probability measure $\mathbb Q$ such that
$$W_t^{\mathbb Q} = W_t-\int_0^t \left(\Gamma^{(1)}_s \frac{1}{M^c_s\vee c}\theta_s-\theta_s- \frac{1}{M^c_s\vee c}U^c_s \right)ds$$
is a Brownian motion under $\mathbb Q$.

Hence,
$$M_s^c=\mathbb E_s^{\mathbb Q}\left[ e^{2\int_s^T r_tdt}+\int_s^T \Gamma^{(1)}_ve^{2\int_s^v r_tdt}|\theta_v|^2 dv\right],$$
from which we deduce that there exists a constant $\eta>0$ independent of $c$ such that $M\ge \eta$. Taking  $c= \eta$, we obtain a solution.

Let us now prove the uniqueness. First we note that if $(M,U)\in L_{\cF}^\infty(0,T;\R)\times L_{\cF}^2(0,T;\R^d)$ is a  solution and there exists $c>0$ such that $M\ge c$, then $U\cdot W$ is a BMO martingale. Let us define
$$Y_s=M_s^{-1},\quad Z_s=-M_s^{-2}U_s.$$
Then $(Y,Z)$ is a solution in  $L_{\cF}^\infty(0,T;\R)\times L_{\cF}^2(0,T;\R^d)   $  of the following BSDE
\begin{equation}\label{YZ}
\left\{\begin{array}{l}
        dY_s=-[-2r Y_s-Z_s'\theta_s-\Gamma^{(1)}_s|\theta_s|^2 Y_s^2+ \Gamma^{(1)}_s Y_sZ_s'\theta]ds+Z_s' dW_s,\\
Y_T=1.\end{array}\right.
\end{equation}
Moreover, $Z\cdot W$ is a BMO martingale.

It suffices to prove uniqueness of solution to BSDE (\ref{YZ}). For this, let $(Y^{(1)},Z^{(1)})$ and $(Y^{(2)},Z^{(2)})$ be two solutions in  $L_{\cF}^\infty(0,T;\R)\times L_{\cF}^2(0,T;\R^d)   $ such that
$Z^{(1)}\cdot W$ and $Z^{(2)}\cdot W$ are BMO martingales. Set
$$\bar{Y}=Y^{(1)}-Y^{(2)},\quad \bar{Z}=Z^{(1)}-Z^{(2)}.$$
Then
\begin{equation}\label{Y'Z'}
\left\{\hspace{-0.2cm}\begin{array}{l}
        d\bar{Y}_s=-[-2r_s \bar{Y}_s- \bar{Z}_s'\theta_s -\Gamma^{(1)}_s |\theta_s |^2(Y^{(1)}_s+Y^{(2)}_s)\bar{Y}_s
        +\Gamma_s^{(1)}\theta_s ' (\bar{Y}_sZ^{(1)}_s
+Y^{(2)}_s \bar{Z}_s)]ds+\bar{Z}_s' dW_s,\\
\bar{Y}_T=0.\end{array}\right.
\end{equation}

Applying Ito's formula to $|\bar{Y_s}|^2$ and taking conditional expectation, we deduce (where $C>0$ is a constant which may change from line to line).
\begin{eqnarray*}
& &|\bar{Y_s}|^2+\mathbb E_s\left[\int_s^T |\bar{Z}_r|^2 dr\right]\\
&\le & C \mathbb E_s \left[\int_s^T |\bar{Y}_r| (|\bar{Y}_r|+ |\bar{Z}_r| +|Z^{(1)}_r||\bar{Y}_r|)dr\right]\\
&\le& C \mathbb E_s\left[\int_s^T |\bar{Y}_r|^2 dr\right]+\frac{1}{2} \mathbb E_s\left[\int_s^T |\bar{Z}_r|^2 dr\right]+C\sqrt{\mathbb E_s\left[\int_s^T |Z^{(1)}_r|^2 dr\right] \mathbb E_s\left[\int_s^T |\bar{Y}_r|^4 dr\right]}\\
&\le& C \mathbb E_s\left[\int_s^T |\bar{Y}_r|^2
dr\right]+\frac{1}{2} \mathbb E_s\left[\int_s^T |\bar{Z}_r|^2
dr\right]+C\sqrt{\mathbb E_s\left[\int_s^T |\bar{Y}_r|^4
dr\right]}.
\end{eqnarray*}
Let us assume that $s\in [T-\delta,T]$. Then by setting
$$\bar{Y}^*_{T-\delta,T}=\|\bar{Y}_\cdot\|_{L_\cF^\infty(T-\delta, T; \R)},$$
we obtain
$$|\bar{Y}_s|^2\le C (\delta+\delta^{1/2}) |\bar{Y}^*_{T-\delta,T}|^2.$$
Hence,
$$|\bar{Y}^*_{T-\delta,T}|^2\le C\delta^{1/2}|\bar{Y}^*_{T-\delta,T}|^2.$$
By taking  $\delta$ sufficiently small, we deduce that $\bar{Y}^*_{T-\delta,T}=0$. We conclude the proof of uniqueness by continuing on $[T-2\delta,T-\delta],\dots$, until time 0 is reached.
\end{proof}

Then we consider the BSDE satisfied by $(\Gamma^{(2)},\gamma^{(2)})$:
\begin{equation}\label{gamma}
\left\{\begin{array}{l}
        d\Gamma^{(2)}_t=-\left[r_t\Gamma^{(2)}_t-\left(\theta_t+\frac{U_t}{M_t}\right)'\gamma^{(2)}_t-\left(|\theta_t|^2+\frac{U_t'\theta_t}{M_t}\right)\Gamma_t \right]dt+(\gamma^{(2)}_t )'dW_t,\\
\Gamma^{(2)}_T=-\mu_2.
\end{array}\right.
\end{equation}

\begin{proposition}
 BSDE (\ref{gamma}) admits a unique solution $(\Gamma^{(2)},\gamma^{(2)})\in L^\infty_{{\cal F}}(0,T;\mathbb R)\times
L^2_{{\cal F}}(0,T;\mathbb R^d)$. Moreover, $\gamma^{(2)}\cdot W$ is a BMO martingale.
\end{proposition}
\begin{proof}
As $-(\theta+\frac{U}{M})\cdot W)$ is a BMO martingale, it
suffices to apply the result of Section 3 in \cite{BC} to deduce
that BSDE (\ref{gamma}) admits a unique solution
$(\Gamma^{(2)},\gamma^{(2)})\in L^2_{{\cal F}}(0,T;\R)\times L^2_{{\cal F}}(0,T;\R^d)$. Let $\Q$ be the
probability measure defined by $\frac{d\Q}{dP}={\cal
E}_T(-(\theta+\frac{U}{M})\cdot W)$. Then under $\Q$,
$$W^\Q_t=W_t+\int_0^t (\theta_s+M_s^{-1}U_s)ds$$
is a Brownian motion and $U\cdot W^\Q$ is a BMO martingale.
Furthermore,
$$d\Gamma^{(2)}_t=-\left[r_t\Gamma^{(2)}_t-\left(|\theta_t|^2+\frac{U_t'\theta_t}{M_t}\right)\Gamma_t\right]dt+(\gamma_t^{(2)})'
dW^\Q_t,\qquad \Gamma^{(2)}_T=-\mu_2.$$
Hence
$$\Gamma^{(2)}_t=\mathbb E^\Q_t\left[-e^{\int_t^T r_vdv}\mu_2-\int_t^T e^{\int_t^s r_vdv}\Gamma_s\left(|\theta_s|^2+\frac{U'_s\theta_s}{M_s}\right)ds\right].$$
From this we deduce that $\Gamma^{(2)}$ is a bounded process. Moreover, from (\ref{gamma}),
\begin{eqnarray*}
 \mathbb E^\Q_t\left[\int_t^T |\gamma^{(2)}_s|^2ds\right]&=&\mathbb E^\Q_t\left[\left|\int_t^T( \gamma^{2}_s)'dW^\Q_s\right|^2\right]\\
&=&\mathbb
E^\Q_t\left[\left|\Gamma^{(2)}_T-\Gamma^{(2)}_t+\int_t^T
\left[r_s\Gamma^{(2)}_s-\Gamma_s\left(|\theta_s|^2+\frac{U'_s\theta_s}{M_s}\right)\right]ds\right|^2\right].
\end{eqnarray*}
Hence from the last equality, $\gamma^{(2)}\cdot W^\Q$ is a BMO
martingale under $\Q$ and then $\gamma^{(2)}\cdot {W}$ is a BMO
martingale under $\p$.
\end{proof}

\medskip

With $M, U, \gamma^{(2)}$ obtained, we can construct a  (feedback) strategy
\begin{equation}\label{feed}
u^*_s=\alpha_s X_s^*+\beta_s
\end{equation}
where $$\alpha_s\trieq \frac{\Gamma^{(1)}_s\theta_s-U_s}{M_s},\quad
\beta_s\trieq -\frac{\Gamma_s\theta_s+\gamma^{(2)}_s}{M_s}.$$

In order to confirm that the above is indeed an {\it admissible} feedback strategy, we need to prove the following technical
result. Its proof is intriguing in its own right.
\begin{proposition}
Let $X^*$ be the solution to the first equation of (\ref{fbsderand}) where $u^*$ is substituted by (\ref{feed}). Then
$X^*\in L^2_{\cal F}(0,T;C(0,T;\R))$ and $u^*\in L^2_{\cal F}(0,T;\R^d)$.
\end{proposition}

\begin{proof}
Plug the feedback strategy $u^*$ into the wealth equation (\ref{wealth2}), we get
\begin{eqnarray}\label{formulaX}
 X_t^*&=&\rho_t\Big(x_0-\int_0^t \rho_s^{-1}\alpha_s'\beta_sds+\int_0^t \rho_s^{-1}\beta_sdW_s^\theta\Big),
\end{eqnarray}
with $W^\theta_t=W_t+\int_0^t\theta_sds$ and $\rho_t=e^{\int_0^tr_sds} {\cal E}_t(\alpha\cdot W^\theta)$.

On the one hand,
\begin{eqnarray*}
{\cal E}_t(\alpha\cdot W^\theta)&=&e^{-\int_0^t \frac{|\alpha_s|^2}{2}ds+\int_0^t\alpha_s'(dW_s+\theta_sds)}\\
&=&e^{-\int_0^t \frac{|\alpha_s|^2}{2}ds-\int_0^t\frac{U_s'}{M_s}dW^\theta_s+\int_0^t\frac{\Gamma^{(1)}_s|\theta_s|^2}{M_s}ds +\int_0^t\frac{\Gamma^{(1)}_s\theta_s'}{M_s}dW_s}\\
&=&e^{-\int_0^t
\left[\frac{|\alpha_s|^2}{2}-\frac{1}{2}\left(\frac{\Gamma^{(1)}_s\theta_s}{M}\right)^2-\frac{\Gamma^{(1)}_s|\theta_s|^2}{M_s}
\right]ds-\int_0^t\frac{U_s'}{M_s}dW^\theta_s}{\cal
E}_t\left(\frac{\Gamma^{(1)}\theta}{M}\cdot W\right).
\end{eqnarray*}
Applying Ito's formula to $\ln(M)$, we get
\begin{eqnarray*}
d\ln(M_s)&=& [-2r_s+\frac{U_s'\theta_s}{M_s}-\Gamma^{(1)}_s
\frac{|\theta_s|^2}{M_s}
+\frac{1}{2}\frac{|U_s|^2}{M_s^2}-\Gamma^{(1)}_s\frac{U_s'\theta_s}{M_s^2}]ds+\frac{U_s'}{M_s}dW_s\\
&=&\left[-2r_s+\frac{|\alpha_s|^2}{2}-\frac{1}{2}\left|\frac{\Gamma^{(1)}_s\theta_s}{M_s}\right|^2-\Gamma^{(1)}_s\frac{|\theta_s|^2}{M_s}\right]ds+\frac{U_s'}{M_s}dW_s^\theta.
\end{eqnarray*}
Combining the above equations, we obtain
$${\cal E}_t(\alpha\cdot W^\theta)=\frac{M_0}{M_t}{\cal E}_t\left(\frac{\Gamma^{(1)}\theta}{M}\cdot W\right)e^{-2\int_0^t r_sds}$$
or
$$\rho_t=\frac{M_0}{M_t}{\cal E}_t\left(\frac{\Gamma^{(1)}\theta}{M}\cdot W\right)e^{-\int_0^t r_sds}.$$
By the fact that $M$ and $\frac{1}{M}$ are both bounded and
$\E{\sup_{t\in [0,T]} \left|{\cal
E}_t\left(\frac{\Gamma^{(1)}\theta}{M}\cdot
W\right)\right|^p}<+\infty$ for any $p\in \R$, we have
$\E{\sup_{t\in [0,T]}\rho_t^p}<+\infty$ for any $p\in \R$.

Now we validate  $X^*\in L_\cF^2(\Omega, C(0,T;\R))$ using (\ref{formulaX}). For
any $p>1$,
\begin{eqnarray*}
&&\E{\sup_{t\in [0,T]} \left| \int_0^t \rho_s^{-1}\alpha_s'\beta_s
ds\right|^p}\\
&\le&\E{\sup_{t\in [0,T] }\rho_t^{-p} \left(  \int_0^T|\alpha_s|^2ds+\int_0^T|\beta_s|^2ds\right)^p}\\
&\le&c_p\sqrt{\E{\sup_{t\in [0,T] }\rho_t^{-2p}}\left(\E{\left(  \int_0^T|\alpha_s|^2ds\right)^{2p}} +\E{\left(\int_0^T|\beta_s|^2ds\right)^{2p}}\right)}\\
&<&+\infty.
\end{eqnarray*}
Similarly we have $\E{\sup_{t\in [0,T]} \left| \int_0^t \rho_s^{-1}\theta_s'\beta_sds\right|^p}<+\infty$. Also we have
\begin{eqnarray*}
\E{\sup_{t\in [0,T]} \left| \int_0^t \rho_s^{-1}\beta_s
dW_s\right|^{2p}}
&\le&\bar c_p\E{\left(\int_0^T\rho_s^{-2}|\beta_s|^2 ds\right)^p}\\
&\le&\bar c_p\E{\sup_{t\in [0,T]} \rho_t^{-2p} \left(\int_0^T|\beta_s|^2 ds\right)^p}\\
&<&+\infty,
\end{eqnarray*}
where $c_p,\bar c_p$ are both constants only depending on $p$. These two inequalities lead to $X^*\in L^2_\cF(\Omega; C(0, T; \R))$.

Finally, regarding $(X^*_\cdot, u^*_\cdot)$ as the solution to the BSDE
\begin{equation}
 \left\{\begin{array}{l}
         dX_s=r_s X_sds+\theta_s'u_sds+u_s'dW_s,\qquad s\in [0,T],\\
     X_T=X^*_T.
        \end{array}
 \right.
\end{equation}
By the standard estimates for Lipschitz BSDE, $u^*\in L_\cF^2(0,T;\R^d)$
as soon as $X^*\in L_\cF^2(\Omega, C(0,T;\R))$.
\end{proof}

\subsection{Equilibrium Strategy}
Summarizing the preceding analysis, we obtain finally the main result of this section.
\begin{theorem}\label{randmain}
Let $(M,U)$ and $(\Gamma^{(2)},\gamma^{(2)})$ be the solutions to BSDEs (\ref{MU}) and (\ref{gamma}) respectively, and $\Gamma_s=-\mu_2 e^{\int_s^T r_t dt}$. Then
$$u^*_s=-M_s^{-1}\left[(U_s-\theta_s \mu_1 e^{\int_s^Tr_vdv}) X^*_s+\Gamma_s\theta_s+\gamma^{(2)}_s\right]$$
is an equilibrium strategy.
\end{theorem}
\begin{proof}
Define  $p, k$ by (\ref{prand}) and (\ref{krand}) (recall that
$N=M$ and $V=U$). It is easy to check that $u^*, X^*, p, k$
satisfies (\ref{fbsderand}). Furthermore,  $\Lambda$ in the
condition (\ref{Lambdacond}) is
\begin{eqnarray*}
\Lambda(s;t)&=&p(s;t)\theta_s+k(s;t)\\
&=&\hspace{-0.3cm}(M_sX^*_s-\Gamma^{(1)}_sX^*_t+\Gamma^{(2)}_s-\Et{M_sX^*_s+\Gamma^{(3)}_s})\theta_s
 +X_s^*U_s+M_s u_s^* +\gamma^{(2)}_s-\gamma^{(1)}_sX^*_t\\
&=&(M_sX^*_s+\Gamma^{(3)}_s-\Et{M_sX^*_s+\Gamma^{(3)}_s})\theta_s
       +\Gamma^{(1)}_s(X^*_s-X^*_t)\theta_s.
\end{eqnarray*}
Since $M, \theta, \Gamma^{(3)}, \Gamma^{(1)}$ are all essentially bounded,
$\Et{\sup_{s\in [t,T]}(X^*_s)^2}<+\infty$, we deduce  that
$\Lambda$ meets condition (\ref{Lambdacond}).  It follows from  Theorem
\ref{main} that $u^*$ is an equilibrium.
\end{proof}

\subsection{Examples}
Equilibrium strategies for mean--variance models have been studied  in \cite{Basak,BM,BMZ} among others in different
frameworks.  In this subsection,  we will compare our results with some existing ones in literature.

\subsubsection{Deterministic risk premium}

Let us first consider the case when the risk premium is deterministic function of time.
Then $U=0$, $\gamma^{(2)}=0$, and
$$M_t=e^{2\int_t^T r_vdv}\left( 1+\mu_1 \int_t^T e^{-\int_s^T r_vdv}|\theta_s|^2ds\right).$$
The equilibrium strategy is given by
$$u^*_t=\frac{\mu_1 e^{\int_t^T r_vdv}}{M_t}\theta_tX_t^*+\frac{\mu_2 e^{\int_t^T r_vdv}}{M_t}\theta_t.$$

\noindent {\bf Case 1:  $\mu_1=0$.}

When $\mu_1=0$, the objective is exactly the same as in \cite{Basak} and \cite{BM}, in which
the equilibrium is however defined within the class of (deterministic) feedback controls.

By Theorem \ref{randmain},
$$u_t^*=e^{-\int_t^T r_vdv}\mu_2\theta_t$$
is a mean-variance equilibrium strategy.
This equilibrium coincides with the one obtained in \cite{Basak} and \cite{BM} although
the definitions of equilibrium are different. The {\it ex-post}  reason is that the feedback part
of our  equilibrium is absent, and so is the gap between the two definitions.

\bigskip

\noindent {\bf Case 2: $\mu_2=0$.}

When $\mu_2=0$, the objective is equivalent to the one in \cite{BMZ}. In this case,
our equilibrium is, explicitly,
$$u_t^*=\frac{\mu_1e^{\int_t^T r_vdv}}{M_t}\theta_t X_t^*.$$

In \cite{BMZ}, the equilibrium is defined for the class of feedback controls as in \cite{BM}. Therein the
equilibrium strategy is derived in a linear feedback form $u^*_t=c_tX^*_t$ with $c_t$ uniquely determined by an integral equation (whose unique solvability
is established). We can easily show that the linear coefficient of our equilibrium above does not satisfy the integral equation
in \cite{BMZ}. This, in turn, indicates the difference between the two definitions of  equilibriums (open-loop and feedback).

\subsubsection{Stochastic risk premium}

When the risk premium  of the market is a stochastic process, the PDE (HJB equation) approach employed by \cite{BM} or \cite{BMZ}, where
the definition of equilibrium is in the class of feedback controls, does no longer work.
To our best knowledge,
our result is the first attempt to formulate and find equilibrium with random market parameters.

\medskip

\noindent {\bf Case 1:  $\mu_1=0$.}

When $\mu_1=0$, $U=0$, $M_t=e^{2\int_t^T r_vdv}$, and our equilibrium is
$$u_t^*=e^{-\int_t^T r_vdv}\mu_2\theta_t- e^{-2\int_t^T r_vdv}\gamma^{(2)}_t.$$
This strategy consists of two parts. The first part is in the same form as
that in the deterministic risk premium case, and the second part is to hedge the uncertainty arising from the randomness
of $\theta$.

\medskip

\noindent {\bf Case 2: $\mu_2=0$.}

When $\mu_2=0$, $\gamma^{(2)}=0$, and our equilibrium is
$$u_t^*=\left(\frac{\mu_1e^{\int_t^T r_vdv}}{M_t}\theta_t-\frac{U_t}{M_t}\right)X_t^*.$$

The linear feedback coefficient in this  equilibrium also consists of two parts.
The first part is formally the same as its deterministic counterpart, whereas the
second part is for the randomness of the parameter $\theta$.

\section{Concluding Remarks}

This paper, we believe, has posed more questions than answers. The flow of FBSDEs (\ref{fbsdegeneral}) is an interesting class of equations, whose
general solvability begs for systematic  investigations. How to adapt the generalized HJB approach of \cite{BM,BMZ} to our open-loop control
framework, even when all the coefficients are
deterministic, warrants a careful study (but notice the fundamental difference in the definitions of equilibrium). Extension beyond the realm of
LQ may open up an entirely new avenue for stochastic control. Finally, how our game theoretic formulation may be extended to other types of
time-inconsistency, e.g., that caused by probability distortion, promises to be an equally exciting research topic. The research on the last 
problem is in progress and will appear in a forthcoming paper.

\end{document}